\documentclass[12pt,a4paper,oneside,openright]{amsart}

\usepackage[a4paper,top=2cm,bottom=2cm,left=2cm,right=2cm]{geometry}
\date{}
\usepackage[utf8]{inputenc}
\usepackage{amsmath}
\usepackage{amsfonts}
\usepackage{amssymb}
\usepackage{amsthm}
\usepackage{esint}
\usepackage{graphicx}
\usepackage[english]{babel}
\usepackage[T1]{fontenc}

\usepackage{graphicx}
\usepackage{xcolor}
\graphicspath{{./Images/}}

\newtheorem{thm}{Theorem}[]
\newtheorem{cor}[thm]{Corollary}
\newtheorem{lem}[thm]{Lemma}
\newtheorem{prop}[thm]{Proposition}

\theoremstyle{definition}
\newtheorem*{mydef}{Definition}

\theoremstyle{remark}

\theoremstyle{remark}
\newtheorem{rmk}{Remark}

\DeclareMathOperator*{\Reg}{\text{Reg}}

\def\Xint#1{\mathchoice 
	{\XXint\displaystyle\textstyle{#1}}%
	{\XXint\textstyle\scriptstyle{#1}}%
	{\XXint\scriptstyle\scriptscriptstyle{#1}}%
	{\XXint\scriptscriptstyle\scriptscriptstyle{#1}}%
	\!\int} 
\def\XXint#1#2#3{{\setbox0=\hbox{$#1{#2#3}{\int}$} 
		\vcenter{\hbox{$#2#3$}}\kern-.5\wd0}} 
\def\fint{\Xint -}
\usepackage{fancyhdr}

\title{Asymptotically quasiconvex functionals with general growth conditions.}
\begin{document}
	
	\maketitle 
	
	\begin{center}
		FRANCESCA ANGRISANI \footnote{ \noindent Laboratoire Jacques Louis Lions (LJLL), Sorbonne, Rue Jussieu 4, Paris, France  \\ francesca.angrisani@sorbonne-universite.fr \\}
		
	\end{center}\begin{abstract}
	We establish local regularity results for minimizers of autonomous vectorial integrals of Calculus of Variations, assuming $\psi$-growth conditions and imposing $\varphi$-quasiconvexity only in an asymptotic sense, both in the sub-quadratic and super-quadratic case. In particular, we obtain $C^{1,\alpha}$ regularity at points $x_0$ where $Du$ is sufficiently large in a neighborhood of $x_0$, as well as Lipschitz regularity on a dense set. \ The results hold for all pairs of Young functions $(\varphi, \psi)$ satisfying the $\Delta_2$ condition. 
\end{abstract}


\section{Introduction}
In this paper, we study multidimensional variational integrals of the form
$$
\mathcal{F}(u) = \int_\Omega f(Du(x))\,dx \quad \text{ for } u: \Omega \to \mathbb{R}^N
$$
where $\Omega$ is an open bounded subset of $\mathbb{R}^n$ with $n \geq 2$ and $N \geq 1$. \\
We consider Young functions $\varphi$ and $\psi$ of class $C^1([0,+\infty)) \cap C^2(0,+\infty)$ such that, satisfy the following conditions: \\
\begin{itemize}
	\item[(H.0)] \begin{equation} \varphi, \psi   \in \Bigg\{h \in L^1(\mathbb{R}, \mathbb{R}): \quad \forall t \quad \exists \quad  p_1(h), q_1(h) >0 \quad 	p_1(h) \frac{h'(t)}{t} \leq h''(t) \leq q_1(h) \frac{h'(t)}{t} \Bigg\} \end{equation}
Where the context makes it clear, we will omit the explicit dependence on $h$.
	\item[(H.1)] \textbf{Regularity} - $f \in C^2(\mathbb{R}^{nN}, \mathbb{R})$.
	\item[(H.2)] \textbf{Asymptotic $W^{1,\varphi}$-quasiconvexity} \cite{Diening} - There exist $M \gg 0$, $\gamma > 0$, and a continuous function $g \in W^{1,\varphi}(\mathbb{R}^{nN}, \mathbb{R})$ such that
	$$
	f(z) = g(z), \quad \forall z: |z| > M
	$$
	and $g$ is strictly $W^{1,\varphi}$-quasiconvex, meaning that it satisfies
	$$
	\fint_{B_1} g(z + D\phi) \geq g(z) + \gamma \fint_{B_1} \varphi_{1+|z|}(|D\phi|), \quad \forall z, \forall \phi \in C_0^\infty(B_1, \mathbb{R}^N)
	$$
	where $\varphi_a(t)$ is defined for any $a > 0$ by
	$$
	\varphi'_a(t) = \varphi'(a+t) \frac{t}{a+t}, \quad \varphi_a(0) = 0
	$$
	and it was shown in \cite{Leone} to satisfy
	$$
	\varphi_a(t) \sim t^2 \varphi''(a+t).
	$$
	\item[(H.3)] \textbf{Growth conditions} - The following inequalities hold:
	$$
	\Gamma'\varphi(|z|) \leq f(z) \leq \Gamma''(1 + \psi(|z|))
	$$
	$$
	|D^2 f(z)| \leq \Gamma''(1 + \psi''(|z|))
	$$
	for all $z \in \mathbb{R}^{nN}$, for some positive constants $\Gamma', \Gamma'' > 0$.
	\item[(H.4)] \textbf{Range of anisotropy} - For any $a > M$, the function $\mathcal{N}_a = \phi_a \circ (\psi_a')^{-1}$ is a Young function, and its complementary Young function $\mathcal{N}_a^*$ satisfies
	$$
	[\mathcal{N}_a]^*(t) \leq c\varphi_a^\beta(t), \quad \forall t \gg 1
	$$
	for some $1 \leq \beta < \frac{n}{n-1}$.
\end{itemize}

From condition $(H.4)$, we deduce the inequality
$$
\psi(t) \leq c \varphi^\beta(t), \quad \forall t \gg 1.
$$
In the special case $\varphi(t) = t^p$ and $\psi(t) = t^q$, condition $(H.4)$ is equivalent to $q < p + \frac{1}{n}$.

In \cite{Leone}, it was proven that condition $(H.2)$ is equivalent to:
\begin{itemize}
	\item[(H.2')] There exist $M \gg 0$, $\gamma > 0$ such that
	$$
	\forall z: |z| > M
	$$
	$$
	\fint_{B_1} f(z + D\phi) \geq f(z) + \gamma \fint_{B_1} \varphi_{1+|z|}(|D\phi|), \quad  \forall \phi \in C_0^\infty(B_1, \mathbb{R}^N).
	$$
\end{itemize}
This follows from the fact that $f$ is locally bounded from below. \\
We study local $W^{1,\varphi}$-minimizers of $\mathcal{F}$, i.e., functions $u$ such that
$$\mathcal{F}(u)\le \mathcal{F}(u+\phi)\quad \forall \phi \in W_0^{1,\varphi}(\Omega,\mathbb{R}^N).$$

In the case of a globally quasiconvex functional and under the super-quadratic hypothesis $\phi(t)>a(t^2-1)$, D. Breit and A. Verde in \cite{Breit} proved that if $u$ is a local minimizer of $\mathcal{F}$, then $u$ belongs to $C^{1,\alpha}$ on an open dense subset of $\Omega$. By adapting and generalizing their arguments, we establish partial $C^{1,\alpha}$ regularity for a local minimizer $u$ of $\mathcal{F}$ around points $x_0$ such that there exists a ball $B_r(x_0)$ where $|Du|$ exceeds $M$.

The study of asymptotic regularity problems has been extensively explored in recent years. Regularity theory for integrals exhibiting a specific structure in the neighborhood of infinity was first investigated by M. Chipot in \cite{Chipot} and later by many others, including T. Isernia, C. Leone, and A. Verde in \cite{Leone}. It is often observed that global quasiconvexity is not necessary; rather, by localizing the natural assumptions at infinity, both in the scalar and vectorial cases, Hölder regularity can be obtained near points where the integral function approaches the value $z_0$ (see also \cite{Leone3}, \cite{Schmidt2}, \cite{Dacorogna}, \cite{Kristensen}, \cite{Simon}, \cite{Raymond}, \cite{Foss}, \cite{Giaquinta}, \cite{Kristensen2}, \cite{Duzaar}).

More precisely, we establish the following result:

\begin{thm} \label{maintheorem2}
	Let $f,\varphi,\psi$ satisfy hypotheses $(H.0)$, $(H.1)$, $(H.2)$, $(H.3)$, and $(H.4)$, and let $u$ be a local minimizer of the corresponding functional $\mathcal{F}$. Suppose that $z_0\in \mathbb{R}^{nN}$ satisfies $|z_0|>M+1$ and that there exists $x_0 \in \mathbb{R}^n$ such that
	$$\lim\limits_{\rho\to 0^+}\fint_{B_\rho(x_0)} \left|V(Du(x))-V(z_0)\right|^2=0,$$
	then $x_0 \in \Reg(u)$, where $\Reg(u)=\{x \in \Omega: \ u \text{ is locally } C^{1,\alpha}(\Omega) \text{ around } x\}$, and $V(z)$ is defined in Section \ref{sezorlicz}.
\end{thm}

This theorem leads to the following interesting corollary:

\begin{cor}
	Under the hypotheses and notation of Theorem \ref{maintheorem2}, the set of points where $u$ is locally Lipschitz continuous is a dense open subset of $\Omega$.
\end{cor}

In the power case, the common approach in this context is to use the blow-up argument. However, in the general growth case, homogeneity is absent, which is crucial in the blow-up method. Therefore, we employ the so-called $\mathcal{A}$-harmonic approximation, as proved in \cite{Diening}. This approach enables us to compare the solutions of our problem with those of the regularized problem in terms of gradient closeness.

A partial regularity result for our type of integrals (i.e., the minimizers are Lipschitz continuous on an open and dense subset of $\Omega$) was previously obtained in \cite{Schmidt2}.
Our result is optimal in the following sense: it is not possible to establish regularity outside a negligible set. Thus, the fact that our minimizers are $C^{1,\alpha}$ on a dense subset of $\Omega$ is optimal in this respect.

\section{Young Functions and Their Properties}

\begin{mydef}
	A real function $\varphi:[0,+\infty)\to [0,+\infty)$ is called a \textbf{Young function} if it satisfies the following conditions:
	\begin{enumerate}
		\item $\varphi(0)=0$.
		\item $\varphi$ is differentiable, and its derivative $\varphi'$ is right-continuous, non-decreasing, and satisfies 
		$$\varphi'(0)=0,\quad \varphi'(t)>0 \text{ for all } t>0.$$
		\item $\lim\limits_{t\to +\infty} \varphi'(t) = +\infty$.
	\end{enumerate}
\end{mydef}

In this paper, we also assume that Young functions are $C^2(0,\infty) \cap C^1([0,+\infty))$, which is not a particularly restrictive assumption. It is easy to see from the definition that a Young function is necessarily convex.

\begin{mydef}
	A Young function $\varphi$ satisfies the $\Delta_2$ condition if there exists a positive constant $k_1 > 0$ such that, for all $t > 0$, the following holds:
	$$ \varphi(2t) \le k_1 \varphi(t), $$
	and we define the \textbf{$\Delta_2$-constant} of $\varphi$ as:
	$$ \Delta_2(\varphi) := \sup\limits_{t>0} \frac{\varphi(2t)}{\varphi(t)}. $$
	
\end{mydef}

If $\varphi$ satisfies the $\Delta_2$ condition, then for any $a > 1$, we have the asymptotic equivalence $\varphi(at) \sim \varphi(t)$.

\begin{mydef}
	The space $L^\varphi$ of functions is defined by
	$$ L^\varphi(\Omega) := \{ f : \Omega \to \mathbb{R} \text{ measurable such that } \int_\Omega \varphi(|f|)\, dx < +\infty \}. $$
	
	This space is called the \textbf{Orlicz space} associated with the Young function $\varphi$.\\
	The \textbf{Orlicz-Sobolev space} $W^{1,\varphi}$ is the space of functions in $L^\varphi$ whose weak derivative is also in $L^\varphi$.\\
	Moreover, by $W_0^{1,\varphi}$ we denote the closure in $W^{1,\varphi}$ of the space of $C^\infty$ functions with compact support.
\end{mydef}

For any Young function $\varphi$, since $\varphi'$ is non-decreasing, the generalized inverse is well-defined:
$$ (\varphi')^{-1}(t) := \sup\{ s \in [0, +\infty) : \varphi'(s) \le t \}, $$
which allows us to define the complementary function $\varphi^*$, also a Young function, implicitly by:
$$ \varphi^*(t)' = (\varphi')^{-1}, \quad \varphi^*(0) = 0. $$

\subsection{Examples}

An example of a family of functions $f$ that satisfy hypotheses $(H.0)$ through $(H.4)$ is given by:
$$ f(z) := \tilde{f}(|z|), \quad \text{where} \quad \tilde{f}(t) := t^p \log^\alpha(1+t)[\sin^2(t) + \cos^2(t)t^\beta], $$
with $0 \le \alpha, 0 \le \beta < \frac{1}{n}$, and the choice of $\varphi$ and $\psi$ as:
$$ \varphi := t^p \log^\alpha(1+t), \quad \psi := t^{p+\beta} \log^\alpha(1+t), $$
which reduces to simpler examples if either $\alpha = 0$ or $\beta = 0$.\\
For more examples of Orlicz spaces, particularly those generated by functions satisfying the $\Delta_2$ condition, see \cite{AngrisaniOrlicz}. 

\section{Technical Lemmas and Definitions} \label{sezorlicz}

The following lemma is a technical tool used to prove the Caccioppoli estimate (for the proof, see \cite{Fonseca}):

\begin{lem}\label{Fonseca2}
	Let $-\infty < r < s < +\infty$, and let $\Xi : [r, s] \to \mathbb{R}$ be a continuous, non-decreasing function. Then there exist $\tilde{r} \in [r, \frac{2r+s}{3}]$ and $\tilde{s} \in [\frac{r+2s}{3}, s]$ such that:
	$$ \frac{\Xi(t) - \Xi(\tilde{r})}{t - \tilde{r}} \le 3 \frac{\Xi(s) - \Xi(r)}{s - r}, $$
	and
	$$ \frac{\Xi(\tilde{s}) - \Xi(t)}{\tilde{s} - t} \le 3 \frac{\Xi(s) - \Xi(r)}{s - r}, $$
	for every $t \in (\tilde{r}, \tilde{s})$.\\
	In particular, we have:
	$$ \frac{s - r}{3} \le \tilde{s} - \tilde{r} \le s - r. $$ 
\end{lem}

The following lemma concerns Young functions satisfying hypothesis $(H.0)$:

\begin{lem} \label{lemmasulleyoung}
	Let $h$ be a Young function satisfying $(H.0)$. Then the following hold:
	\begin{itemize}
		\item[(a)] $h$ satisfies $\Delta_2(h) < +\infty$ and $\Delta_2(h^*) < +\infty$.
		\item[(b)] For all $t > 0$, the following inequality holds:
		$$ h(1)(t^p - 1) \le h(t) \le h(1)(t^q + 1), $$
		where $p = p_1(h) + 1$ and $q = q_1(h) + 1$.
		\item[(c)] For all $t > 0$, the expression $h'(t)t$ is equivalent to $h(t)$.
	\end{itemize}
\end{lem}

For the proof, see Lemma 3.1 in \cite{Fey}. 

\begin{mydef}[Excess]
	For any $z \in \mathbb{R}^{nN}$, we define the quantity
	$$ V(z) := \sqrt{\frac{\varphi'(|z|)}{|z|}} z, $$
	and note that, under our hypotheses, we have
	$$ |V(z_1) - V(z_2)|^2 \simeq \varphi_{|z_1|}(|z_1 - z_2|). $$
	
	We also define the excess function
	$$ \Phi_\varphi(u, x_0, \rho, z) := \fint_{B_\rho(x_0)} |V(Du) - V(z)|^2 \, dx $$
	and
	$$ \Phi_\varphi(u, x_0, \rho) := \fint_{B_\rho(x_0)} |V(Du) - V[(Du)_{B_\rho(x_0)}]|^2 \, dx, $$
	where by placing a set as a subscript to a function, we refer to the integral average of the function over the set. That is, 
	$$ [V(Du)]_{B_\rho(x_0)} = \fint_{B_\rho(x_0)} V(Du) \, dx. $$
	
	We immediately observe that
	\begin{equation} \label{equiv}
		\Phi_\varphi(u, x_0, \rho, z) \simeq \fint \varphi_{|z|}(|Du - z|) \, dx.
	\end{equation}
\end{mydef}


\begin{rmk}
	Under hypothesis $(H.3)$, assumptions $(H.2)$ and $(H.2')$ are known to be equivalent.
\end{rmk}

\section{Outline of the Proof of the Main Theorem and Preliminary Lemmas}

In \cite{Breit}, D. Breit and A. Verde proved that if $u$ is a $W^{1,\varphi}$-minimizer of $\mathcal{F}$ on $B_\rho(x_0)$, for all $L > 0$ and $\alpha \in (0, 1)$, there exists $\varepsilon_0 > 0$ such that if 
$$ \Phi_\varphi(u, x_0, \rho) \le \varepsilon_0 \quad \text{and} \quad \left|\fint_{B_\rho(x_0)} Du\right| \le \frac{L}{2}, $$ 
then $u \in C^{1,\alpha}_{\text{loc}}(B_\rho(x_0); \mathbb{R}^n)$.\\

We will now replicate their reasoning under the weaker hypothesis of asymptotic $\varphi$-quasiconvexity and without assuming the superquadratic growth behavior of $\varphi$.\\
Once we have proved the Caccioppoli inequality and used $\mathcal{A}$-harmonic approximation, it is essential to obtain growth estimates for the excess, because the perturbation terms from the Caccioppoli inequality can be controlled using the smallness of the excess.\\
Using Campanato's integral characterization of Hölder continuity (see \cite{Campanato}), we will be able to prove the main result.\\

Let us begin by proving the following preliminary lemma:

\begin{lem} \label{puntiregolarivarphi}
	If there exists $z_0$ with $|z_0| > M + 1$ and $x_0$ such that
	$$ \fint_{B_\rho(x_0)} \left|V(Du) - V(z_0)\right|^2 \to 0 \quad \text{as} \quad \rho \to 0^+, $$ 
	then there exists $r_1 = r_1(x_0, z_0)$ such that for all $r < r_1$,
	$$ \left|\fint_{B_r(x_0)} Du\right| > M + 1. $$
\end{lem}

\begin{proof}
	Let $|z_0| = M + 1 + \varepsilon$. Then, by the definition of the limit, there exists a $r_1$ (depending on $x_0$ and $z_0$) such that for all $r < r_1$, we have, thanks to \eqref{equiv}, that
	$$ \fint_{B_r(x_0)} \varphi_{|z_0|}(|Du - z_0|) \le \varphi_{|z_0|}\left(\frac{\varepsilon}{2}\right), \quad \forall r < r_1. $$
	\\
	Using Jensen's inequality, we also obtain that
	$$ \left|\fint_{B_r(x_0)} Du - z_0 \right| \le \frac{\varepsilon}{2}, \quad \forall r < r_1, $$
	which gives
	$$ \left|\fint_{B_r(x_0)} Du \right| \ge |z_0| - \frac{\varepsilon}{2} = M + 1 + \varepsilon - \frac{\varepsilon}{2} > M + 1, \quad \forall r < r_1. $$
\end{proof}

Now, we state Lemma 2.5 from \cite{Breit}, which generalizes the extension operator from \cite{Fonseca} to Orlicz spaces. This is a useful tool for us as well:

\begin{lem} \label{duepunticinque}
	Let $0 < r < s$ and $\alpha \ge p$. Then there exists a linear operator 
	$$ T_{r,s}: W^{1,\varphi}(\Omega) \to W^{1,\varphi}(\Omega), $$
	defined as
	$$ T_{r,s} u(x) = \fint_{B_1(0)} u(x + \xi(x) y) \, dy, \quad \text{where} \quad \xi(x) := \frac{\max\{0, \min\{|x| - r, s - |x|\}\}}{2}, $$
	such that the following properties hold:
	\begin{itemize}
		\item[(a)] $T_{r,s} u = u$ on $B_r$ and outside $\overline{B_s}$;
		\item[(b)] $T_{r,s} u \in u + W_0^{1,\varphi}(B_s \setminus \overline{B_r}, \mathbb{R}^n)$;
		\item[(c)] $|DT_{r,s} u| \le c T_{r,s} |Du|$;
		\item[(d)] The following estimates hold:
		\begin{align*}
			\int_{B_s \setminus B_r} \varphi(|T_{r,s} u|) \, dx &\le c \int_{B_s \setminus B_r} \varphi(|u|) \, dx, \\
			\int_{B_s \setminus B_r} \varphi(|DT_{r,s} u|) \, dx &\le c \int_{B_s \setminus B_r} \varphi(|Du|) \, dx, \\
			\int_{B_s \setminus B_r} \varphi^\beta(|T_{r,s} u|) \, dx &\le c(s - r)^{-n\beta + n + \beta} \left[ \sup_{r \le t \le s} \frac{\theta(t) - \theta(r)}{t - r} + \sup_{r \le t \le s} \frac{\theta(s) - \theta(t)}{s - t} \right], \\
			\int_{B_s \setminus B_r} \varphi^\beta(|DT_{r,s} u|) \, dx &\le c(s - r)^{-n\beta + n + \beta} \left[ \sup_{r \le t \le s} \frac{\Theta(t) - \Theta(r)}{t - r} + \sup_{r \le t \le s} \frac{\Theta(s) - \Theta(t)}{s - t} \right].
		\end{align*}
	\end{itemize}
	where
	$$ \theta(t) := \int_{B_t} \varphi(|u|) \, dx, \quad \Theta(t) := \int_{B_t} \varphi(|Du|) \, dx. $$
\end{lem}

The final preliminary lemma provides an estimate that will be useful in proving the Caccioppoli inequality for both subquadratic and superquadratic cases:

\begin{lem} \label{lemmasuH}
	Let \(L\) be any positive constant larger than \(M\), and let \(a\) be a real number in the interval \((M, L)\). Then, for any \(t > 0\), we have:
	\[
	\psi_a(t) \leq K \cdot H(\varphi_a(t)),
	\]
	where \(K = K(M, L, \beta, \psi, \varphi)\) is a positive constant depending on \(L\), and \(H(t) := t + t^\beta\).
\end{lem}

\begin{proof}
	We begin with the case \(t \leq 1\). In this case, by property \((H.2)\), we have:
	\[
	\psi_a(t) \simeq \psi''(a + t)t^2 \simeq \psi(a + t)\frac{t^2}{(a + t)^2} \leq \max\limits_{[M, L+1]}\psi \cdot \frac{t^2}{(a + t)^2} \leq K_1 \min\limits_{[M, L+1]}\varphi \cdot \frac{t^2}{(a + t)^2} \leq K_1 \varphi(a + t) \frac{t^2}{(a + t)^2} \simeq K_1 \varphi_a(t),
	\]
	where
	\[
	K_1 = \frac{\max\limits_{[M, L+1]}\psi}{\min\limits_{[M, L+1]}\varphi} \in (0, +\infty),
	\]
	which depends only on \(M\), \(L\), \(\psi\), and \(\varphi\).
	
	On the other hand, if \(t > 1\), we obtain:
	\[
	\psi_a(t) \simeq \psi''(a + t)t^2 \simeq \psi(a + t) \frac{t^2}{(a + t)^2} \leq \varphi^\beta(a + t) \left( \frac{t^2}{(a + t)^2} \right)^\beta \left( 1 + \frac{a}{t} \right)^{2\beta - 2} \leq K_2 \left[ \frac{\varphi(a + t)t^2}{(a + t)^2} \right]^\beta \simeq K_2 \varphi_a(t)^\beta,
	\]
	where \(K_2 = (1 + L)^{2\beta - 2}\).
	
	The thesis follows with \(K = \max\{K_1, K_2\}\).
\end{proof}

\section{Caccioppoli Inequality}

We are now ready to prove the Caccioppoli inequality. This section is dedicated to the Caccioppoli inequality, which is the main tool used to prove partial regularity of solutions for this type of problem:

\begin{lem} \label{caccioppolilemmabis}
	Let the assumptions $(H.0)-(H.4)$  hold for a given \(M\). Choose any positive constants \(L > M > 0\), and let \(u \in W^{1, \varphi}(B_{\rho}(x_0); \mathbb{R}^N)\) be a minimizer of \(\mathcal{F}\) on the ball \(B_\rho(x_0)\) contained in \(\Omega\). Then, for all \(z \in \mathbb{R}^{nN}\) with \(M < |z| < L+1\), let \(q(x)\) be an affine function with gradient \(z\) and \(v(x) = u(x) - q(x)\). We have:
	\begin{equation} \label{caccioppolibis}
	\fint_{B_{\frac{\rho}{2}}} \varphi_{|z|}(|Dv|) \, dx \leq c \fint_{B_\rho} \varphi_{|z|}\left( \frac{|v|}{\rho} \right) \, dx + c \left\{ \fint_{B_\rho} \left[ \varphi_{|z|}(|Dv|) + \varphi_{|z|}\left( \frac{|v|}{\rho} \right) \right] \, dx \right\}^\beta.
	\end{equation}
\end{lem}

\begin{proof}
	Let us assume, for simplicity, that \(x_0 = 0\), and let \( \frac{\rho}{2} \leq r < s \leq \rho \). Define:
	\[
	\Xi(t) := \int_{B_t} \left[ \varphi_{|z|}(|Dv|) + \varphi_{|z|}\left( \left| \frac{v}{\tilde{s} - \tilde{r}} \right| \right) \right] \, dx.
	\]
	We choose \(r \leq \tilde{r} < \tilde{s} \leq s\) as in Lemma \ref{Fonseca2}. Let \(\eta\) denote a smooth cutoff function with support in \(B_{\tilde{s}}\), satisfying \(\eta \equiv 1\) in \(\overline{B_{\tilde{r}}}\) and \(0 \leq \eta \leq 1\), with \(|\nabla \eta| \leq \frac{2}{\tilde{s} - \tilde{r}}\) on \(B_\rho\). Using the extension operator from Lemma \ref{duepunticinque}, we set:
	\[
	\zeta := T_{\tilde{r}, \tilde{s}}[(1-\eta)v] \quad \text{and} \quad \xi := v - \zeta.
	\]
	By \(W^{1,\varphi}\)-quasiconvexity, we have:
	\[
	\gamma \int_{B_{\tilde{s}}} \varphi_{|z|}(|D\xi|) \leq \int_{B_{\tilde{s}}} f(z + D\xi) - f(z) = \int_{B_{\tilde{s}}} f(z + D\xi) - f(Du) + f(Du) - f(Du - D\xi) + f(Du - D\xi) - f(z).
	\]
	Since \(f(Du) - f(Du - D\xi) \leq 0\) and \(Du = z + D\xi + D\zeta\), we obtain:
	\[
	\int_{B_{\tilde{s}}} f(z + D\xi) - f(Du) + f(Du) - f(Du - D\xi) + f(Du - D\xi) - f(z) \leq \int_{B_{\tilde{s}}} f(z + D\xi) - f(z + D\xi + D\zeta) + \int_{B_{\tilde{s}}} f(z + D\xi) - f(z).
	\]

Let us begin reasoning on \(\mathcal{I}_1\), recalling our growth hypothesis \((H.3)\) and also making use of Lemma 3.2 from \cite{Leone}. We deduce:  
\begin{multline*}  
	\mathcal{I}_1\le  \int_{B_{\tilde{s}}}\int_0^1\int_0^1 |D^2f(tz+(1-t)(z+\theta D\zeta))||\theta D\zeta|D\zeta|\,dt\,d\theta,dx \le \\  
	\le \Gamma''  \int_{B_{\tilde{s}}}\int_0^1\int_0^1 |\psi''(|tz+(1-t)(z+\theta D\zeta|))||D\zeta|^2\,dt\,d\theta,dx \le \\  
	\le c \int_{B_{\tilde{s}}} \frac{\psi'(2|z|+|z+D\zeta|)}{2|z|+|z+D\zeta|}\,dx \le \\  
	\le c \int_{B_{\tilde{s}}} \psi_{|z|}(|D\zeta|)\,dx.  
\end{multline*}  

Regarding \(\mathcal{I}_2\), we deduce:  
\begin{multline*}  
	\mathcal{I}_2\le \int_{B_{\tilde{s}}}\int_0^1\int_0^1 |D^2f(t(z+D\xi+\theta D\zeta)+(1-t)z)||D\xi+\theta D\zeta||D\zeta|\,dt\,d\theta\,dx \le \\  
	\le c  \int_{B_{\tilde{s}}}\int_0^1\int_0^1 |\psi''(t(z+D\xi+\theta D\zeta)+(1-t)z)||D\xi+\theta D\zeta||D\zeta|\,dt\,d\theta\,dx \le \\  
	\le c \int_{B_{\tilde{s}}} \psi''(|z|+|D\xi|+|D\zeta|)(|D\xi|+|D\zeta|)|D\zeta|\,dx\le \\  
	\le \int_{B_{\tilde{s}}} \psi'_{|z|}(|D\xi|+|D\zeta|)|D\zeta|\,dx \le \\  
	\le c \int_{B_{\tilde{s}}} \psi'_{|z|}(|D\xi|)|D\zeta|+c\int_{B_{\tilde{s}}} \psi'_{|z|}(|D\zeta|)|D\zeta|\le \\  
	\le c \int_{B_{\tilde{s}}} \psi'_{|z|}(|D\xi|)|D\zeta|+c\int_{B_{\tilde{s}}} \psi_{|z|}(|D\zeta|)\,dx.  
\end{multline*}  

By combining our estimates, we obtain:  
\[
\gamma\int_{B_{\tilde{s}}}\varphi_{|z|}(|D\xi|)\le c \int_{B_{\tilde{s}}} \psi'_{|z|}(|D\xi|)|D\zeta|+c\int_{B_{\tilde{s}}\setminus B_{\tilde{r}}} \psi_{|z|}(|D\zeta|)\,dx.
\]  

Applying our anisotropy assumption \((H.3)\) and Lemma \ref{lemmasuH}, we derive the following estimate:  
\begin{multline*}  
	\gamma\int_{B_{\tilde{s}}}\varphi_{|z|}(|D\xi|) \le c\int_{B_{\tilde{s}}} H[\varphi_{|z|}(|D\zeta|)]\,dx+\\  
	+c\left[\int_{B_{\tilde{s}}\setminus B_{\tilde{r}}}\varphi_{|z|}(|D\xi|)\,dx+\int_{B_{\tilde{s}}\setminus B_{\tilde{r}}}\varphi_{|z|}^\beta(|D\zeta|)\,dx\right]\le \\  
	\le c\left[\int_{B_{\tilde{s}}\setminus B_{\tilde{r}}}\varphi_{|z|}(|D\zeta|)\,dx+\int_{B_{\tilde{s}}\setminus B_{\tilde{r}}}\varphi_{|z|}^\beta(|D\zeta|)\,dx+\int_{B_{\tilde{s}}\setminus B_{\tilde{r}}}\varphi_{|z|}(|D\xi|)\right]=\\  
	=  \int_{B_{\tilde{s}}\setminus B_{\tilde{r}}}\varphi_{|z|}(|DT_{\tilde{r},\tilde{s}}[(1-\eta)v]|)\,dx+ \int_{B_{\tilde{s}}\setminus B_{\tilde{r}}} \varphi_{|z|}^\beta(|DT_{\tilde{r},\tilde{s}}[(1-\eta)v]|)+\\  
	+c \int_{B_{\tilde{s}}\setminus B_{\tilde{r}}}\phi_{|z|}(|Dv|) \le \\  
	\le c \int_{B_{\tilde{s}}\setminus B_{\tilde{r}}}\varphi_{|z|}(|D\eta||v|+|Dv|)+c(\tilde{s}-\tilde{r})^{-n\beta+n+\beta}\left[\sup\limits_{[\tilde{r},\tilde{s}]}\frac{\Xi(t)-\Xi(\tilde{r})}{t-\tilde{r}}+\sup\limits_{[\tilde{r},\tilde{s}]}\frac{\Xi(\tilde{s})-\Xi(t)}{\tilde{s}-t}\right]^\beta+\\  
	+ c\int_{B_{\tilde{s}}\setminus B_{\tilde{r}}}\phi_{|z|}(|Dv|)\,dx \le \\  
	\le c' \int_{B_{\tilde{s}}\setminus B_{\tilde{r}}}\phi_{|z|}\left(\left|\frac{v}{\tilde{s}-\tilde{r}}\right|\right)+\varphi_{|z|}(|Dv|)\,dx+\\  
	+ c(s-r)^{-n\beta+n}[\Xi(s)-\Xi(r)]^\beta.  
\end{multline*}  

Here, we have applied point (d) of Lemma \ref{duepunticinque} and Lemma \ref{Fonseca2}.  

Now, proceeding from the left-hand side again, we obtain:  
\begin{multline*}  
	\int_{B_r} \varphi_{|z|}(|Dv|)\,dx \le c \int_{B_\rho} \varphi_{|z|}\left(\frac{|v|}{\tilde{s}-\tilde{r}}\right)+\\  
	+c'\int_{B_s\setminus B_r}\varphi_{|z|}(|Dv|)\,dx+c(s-r)^n\left[\frac{\Xi(\rho)}{(s-r)^n}\right]^\beta.  
\end{multline*}  

Using the hole-filling method, we get:  
\begin{multline*}  
	\int_{B_r} \varphi_{|z|}(|Dv|)\,dx  \le  \frac{c'}{1+c'}\int_{B_s} \varphi_{|z|}(|Dv|)+\\  
	+c(s-r)^n\left[(s-r)^{-n}\int_{B_\rho} \varphi_{|z|}(|Dv|)+\varphi_{|z|}\left(\frac{|v|}{\tilde{s}-\tilde{r}}\right)\right]^{\beta}+\\  
	+c \int_{B_\rho} \varphi_{|z|}\left(\frac{|v|}{s-r}\right)\,dx.  
\end{multline*}  

A well-known lemma by Giaquinta (see \cite{Giaquinta}, Chapter V, Lemma 3.1) concludes the proof.

\end{proof}

\section{$\mathcal{A}$-harmonicity}

Let us consider a bilinear form $\mathcal{A}$ on $\mathbb{R}^{nN}$ and assume that the upper bound  
\begin{equation} \label{limitatezza}
|\mathcal{A}| \leq \Lambda
\end{equation}
holds for some constant $\Lambda > 0$, and that the Legendre-Hadamard condition 
\begin{equation}\label{legendre}
\mathcal{A}(y x^T, y x^T) \geq \lambda |x|^2 |y|^2 \quad \text{for all} \quad x \in \mathbb{R}^n, \, y \in \mathbb{R}^N,
\end{equation}
is satisfied with ellipticity constant $\lambda > 0$. We say that $h \in W^{1,1}_{loc}(\Omega, \mathbb{R}^N)$ is $\mathcal{A}$-harmonic on $\Omega$ if 
\[
\int_\Omega \mathcal{A}(Dh, D\phi) \, dx = 0
\]
holds for all smooth $\phi : \Omega \to \mathbb{R}^N$ with compact support in $\Omega$.

The following lemma guarantees that for large $z$, the bilinear form $\mathcal{A} = D^2 f(z)$ satisfies the Legendre-Hadamard condition.

\begin{lem} \label{lemmanovebis}
Let $f$ satisfy $(H.0)$ and $(H.2')$ for a given $M > 0$. Then, for any $z$ such that $|z| > M$, the bilinear form $\mathcal{A} = D^2 f(z)$ satisfies the Legendre-Hadamard condition:
\[
\mathcal{A}(\zeta x^T, \zeta x^T) \geq \lambda |x|^2 |\zeta|^2 \quad \text{for all} \quad x \in \mathbb{R}^n \quad \text{and} \quad \zeta \in \mathbb{R}^N,
\]
with ellipticity constant $\lambda = 2\gamma$.
\end{lem}
\begin{proof}
Let $u$ be the affine function $u(x) = zx$ with $z$ such that $|z| > M$. $W^{1,\varphi}$-quasiconvexity in $z$ ensures that $u$ is a $W^{1,\varphi}$-minimizer of the functional $\mathcal{F}$ induced by $f$, and the function
\[
G(t) = G_{\Phi}(t) := \mathcal{F}_{B_1}(u + t\Phi) - \gamma \int_{B_1} \varphi_{1 + |z|}(|tD\Phi|) \, dx
\]
has a minimum at $t = 0$ for any $\phi \in W_0^{1,\varphi}(B_1, \mathbb{R}^N)$. Following the same steps as in (\cite{Giusti}, Prop. 5.2), from $G_{\Phi}'(0) = 0$ and $G_{\Phi}''(0) \geq 0$, the Legendre-Hadamard condition follows. 

In fact, from $G''(0) \geq 0$, we obtain:
\begin{equation}\label{dasommare}
\int_{B_1} \frac{\partial^2 F}{\partial z_k^\alpha \partial z_j^\beta}(z_0) D_k \phi^\alpha D_j \phi^\beta \, dx \geq 2\gamma \int_{B_1} |D\phi|^2 \, dx
\end{equation}
for every $\phi \in C_c^1(B_1, \mathbb{R}^N)$.

Let us choose \(\phi = \nu + i\mu\) and write \eqref{dasommare} separately for \(\nu\) and \(\mu\), obtaining:  
\begin{equation}  
	\int_{B_1} \frac{\partial^2 F}{\partial z_k^\alpha\partial z_j^\beta}(z_0)D_k\nu^\alpha D_j\nu^\beta\,dx\ge 2\gamma \int_{B_1} |D\nu|^2 \,dx,  
\end{equation}  
and  
\begin{equation}  
	\int_{B_1} \frac{\partial^2 F}{\partial z_k^\alpha\partial z_j^\beta}(z_0)D_k\mu^\alpha D_j\mu^\beta\,dx\ge 2\gamma \int_{B_1} |D\mu|^2\,dx.  
\end{equation} 

Adding these inequalities, we obtain:  
\begin{equation}  
	\int_{B_1} \frac{\partial^2 F}{\partial z_k^\alpha\partial z_j^\beta}(z_0)\left[D_k\nu^\alpha D_j\nu^\beta+D_k\mu^\alpha D_j\mu^\beta\right]\,dx\ge 2\gamma \int_{B_1} (|D\nu|^2+|D\mu|^2)\,dx.  
\end{equation}  

Thus, we get:  
\[
\text{Re} \int_{B_1} \frac{\partial^2 F}{\partial z_k^\alpha \partial z_j^\beta}(z_0)D_k\phi^\alpha D_j\overline{\phi}^\beta\,dx \ge 2\gamma \int_{B_1} |D\phi|^2 \,dx.  
\]  

Now, let us consider any \(\xi \in \mathbb{R}^n\), \(\eta \in \mathbb{R}^N\), \(\tau \in \mathbb{R}\), and \(\varPsi(x) \in C_c^\infty(B_1,\mathbb{R})\), and define \(\phi\) as  
\[
\phi(x) = \eta e^{i\tau(\xi \cdot x)}\varPsi(x).
\]  
Since \(\phi^\alpha(x) = \eta^\alpha\varPsi(x)e^{i\tau \xi \cdot x}\), we obtain  
\[
\int_{B_1} \frac{\partial^2 F}{\partial z_k^\alpha \partial z_j^\beta}(z_0) \eta^\alpha\eta^\beta [\tau^2\xi_k\xi_j\varPsi^2 + D_k\varPsi D_j\varPsi]\,dx \ge 2\gamma |\eta|^2 \int_{B_1} (|D\varPsi|^2+\tau^2|\xi|^2|\varPsi(x)|^2)\,dx.
\]  

Dividing by \(\tau^2\) and letting \(\tau \to \infty\), we deduce:  
\[
\int_{B_1} \frac{\partial^2 F}{\partial z_k^\alpha \partial z_j^\beta}(z_0) \xi_k\xi_j\eta^\alpha\eta^\beta\varPsi^2(x)\,dx \ge 2\gamma |\eta|^2|\xi|^2\int_{B_1}\varPsi^2(x)\,dx.
\]  

Since this holds for all \(\varPsi \in C_c^\infty(B_1,\mathbb{R})\), the proposition is proved.

\end{proof}
\begin{rmk} \label{modulus}
If $f \in C^2_{\text{loc}}(\mathbb{R}^{nN})$, for each $L > 0$, there exists a modulus of continuity $\omega_L : [0, +\infty) \to [0, +\infty)$ satisfying $\lim_{z \to 0} \omega_L(z) = 0$ such that for all $z_1, z_2 \in \mathbb{R}^{nN}$, we have:
\[
|z_1| \leq L, \quad |z_2| \leq L + 1 \quad \Rightarrow \quad |D^2 f(z_1) - D^2 f(z_2)| \leq \omega_L(|z_1 - z_2|^2).
\]
Moreover, $\omega_L$ can be chosen such that the following properties hold:
\begin{enumerate}
\item $\omega_L$ is non-decreasing,
\item $\omega_L^2$ is concave,
\item $\omega_L^2(z) \geq z$ for all $z \geq 0$.
\end{enumerate}
\end{rmk}
\medskip
This lemma essentially proves that if \( Du \) is close to \( z \), subtracting an affine function with gradient \( z \) from \( u \) results in a function that is "almost" \( D^2f(z) \)-harmonic. 

\begin{lem} \label{diecibis}
	Let \( f \) satisfy conditions \((H.0)\) through \((H.4)\) for a given \( M > 0 \). Let \( L > M > 0 \) and take \( u \in W^{1,\varphi}(\mathbb{R}^n, \mathbb{R}^N) \) to be a \( W^{1,\varphi} \)-minimizer of \( \mathcal{F} \) on some ball \( B_\rho(x_0) \). Then, for all \( z \) such that \( M < |z| \le L \) and \( \phi \in C_c^\infty(B_\rho(x_0)) \), we have the following estimate:
	\begin{equation} \label{questarobaquabis}
	\left| \fint_{B_\rho(x_0)} D^2f(z)(Du-z,D\phi)\,dx \right|\le c\sqrt{\Phi_\varphi}\omega_{L}(\Phi_\varphi)\sup\limits_{B_\rho(x_0)}|D\phi|.
\end{equation}
	where \( \Phi_\varphi := \Phi_\varphi(u, x_0, \rho, z) \), the constant \( c \) depends only on \( n \), \( N \), \( \Gamma' \), \( \Gamma'' \), \( L \), and \( \omega_{L} \) is the modulus of continuity from Remark \ref{modulus} (see also \cite{Schmidt}).
\end{lem}

\begin{proof}
	Let \( v(x) := u(x) - zx \). The Euler equation of \( \mathcal{F} \) gives:
	\[
	\left| \fint_{B_\rho} D^2f(z)(Dv, D\phi) \, dx \right| = \left| \fint_{B_\rho} \left[ D^2f(z)(Dv, D\phi) + Df(z) D\phi - Df(Du) D\phi \right] \, dx \right|.
	\]
	If \( |Dv| \le 1 \), we have:
	\[
	|D^2f(z)(Dv, D\phi) + Df(z) D\phi - Df(Du) D\phi| \le \int_0^1 \left| D^2f(z) - D^2f(z + t Dv) \right| \, dt |Dv| \| D\phi \|_\infty \le \omega_L(|Dv|^2) |Dv| \| D\phi \|_\infty.
	\]
	Since \( \omega_L \) is a modulus of continuity, we get:
	\[
	\le c \omega_L(\varphi_{|z|}(|Dv|)) \varphi_{|z|}(|Dv|) \| D\phi \|_\infty,
	\]
	where we used that \( |Dv|^2 \le \inf_{t \in [M, L+1]} \varphi''(t) |Dv|^2 \le \varphi''(|z| + |Dv|) |Dv|^2 \simeq \varphi_{|z|}(|Dv|) \) from \((H.1)\). 
	If, instead, \( |Dv| > 1 \), then since \( M \le |z| \le L \), the condition \((H.1)\) implies \( \psi'(t) > c t \) for \( t > 1 \) and \((H.3)\) holds, so we obtain:
	\[
	|D^2f(z)(Dv, D\phi) + Df(z) D\phi - Df(Du) D\phi| \le c \left( |Dv| + |Dv| \int_0^1 D^2f(|z + t (Du - z)|) \, dt \right) \| D\phi \|_\infty.
	\]
	This can be bounded by:
	\[
	\le c \left( |Dv| + \int_0^1 \frac{\psi'(|z + t(Du - z)|)}{|z + t(Du - z)|} \, dt \right) \| D\phi \|_\infty \le c \left[ \psi'(|Dv|) + \frac{\psi'(2|z|)}{|z|} + \frac{\psi'(2|Du - z|)}{2|Du - z|} \right] \| D\phi \|_\infty.
	\]
	We can bound this further:
	\[
	\le c \left[ \psi'(|Dv|) + 1 \right] \| D\phi \|_\infty \le c \varphi(|Dv|) \| D\phi \|_\infty \le c \varphi_{|z|}(|Dv|) \| D\phi \|_\infty.
	\]
	Now, since \( \omega_L^2(t) \ge t \) for \( t \ge 0 \), we get:
	\[
	\left| \fint_{B_\rho} D^2f(z)(Dv, D\phi) \, dx \right| \le c \| D\phi \|_\infty \fint_{B_\rho} \omega_L(\varphi_{|z|}(|Dv|)) \sqrt{\varphi_{|z|}(|Dv|)} \, dx.
	\]
	Since \( \omega_L \) is non-decreasing, by using the Cauchy-Schwarz and Jensen inequalities, we obtain:
	\[
	\left| \fint_{B_\rho} D^2f(z)(Dv, D\phi) \, dx \right| \le c \sqrt{\Phi_\varphi} \omega_L(\Phi_\varphi) \sup_{B_\rho(x_0)} |D\phi|,
	\]
	which concludes the proof.
\end{proof}

Now, we give the statement of Theorem 3.3 from \cite{Leone2}, as it is useful for us as well:

\begin{lem} \label{Armonicbis}
	Let \( 0 < \lambda \le \Lambda < \infty \) and \( \varepsilon > 0 \). Then there exists \( \delta(n, N, \varphi, \varphi^*, \Lambda, \lambda, \varepsilon) > 0 \) such that the following assertion holds: For all \( \kappa > 0 \), for all \( \mathcal{A} \) satisfying \eqref{limitatezza} and \eqref{Legendre}, and for each \( u \in W^{1,\varphi}(B_\rho(x_0); \mathbb{R}^N) \) satisfying:
	\[
	\left| \fint_{B_\rho(x_0)} \mathcal{A}(Du, D\phi) \, dx \right| \le \delta \kappa \sup_{B_\rho(x_0)} |D\phi|,
	\]
	for all smooth \( \phi : B_\rho(x_0) \to \mathbb{R}^N \) with compact support in \( B_\rho(x_0) \), there exists an \( \mathcal{A} \)-harmonic function \( h \in C^\infty_{\text{loc}}(B_\rho(x_0), \mathbb{R}^N) \) such that:
	\[
	\sup_{B_{\rho/2}(x_0)} |Dh| + \rho \sup_{B_{\rho/2}(x_0)} |D^2h| \le c^* \varphi_{|z|}^{-1} \left( \fint_{B_\rho(x_0)} \varphi_{|z|}(|Du|) \right),
	\]
	and
	\[
	\fint_{B_{\rho/2}(x_0)} \varphi_{|z|} \left( \frac{|u - h|}{\rho} \right) \, dx \le \varepsilon \left[ \fint_{B_\rho(x_0)} \varphi_{|z|}(|Du|) + \varphi(\gamma) \right].
	\]
	Here, \( c^* \) denotes a constant depending only on \( n, N, q_1, \Lambda, \lambda \).
\end{lem}

\section{Excess Decay Estimate}

\begin{prop} \label{rpfopsemplice}
	Let \( z_0 \) be such that \( |z_0| > M + 1 \) and \( x_0 \) be such that:
	\[
	\lim_{\rho \to 0} \fint_{B_\rho(x_0)} \left| V(Du(x)) - V(z_0) \right|^2 \, dx = 0.
	\]
	Then:
	\[
	\Phi_p(u, x_0, \rho) \to 0 \quad \text{as} \quad \rho \to 0.
	\]
	
	\begin{proof}
		Let \( (Du)_\rho := \fint_{B_\rho(x_0)} |Du| \). By the triangle inequality, we have:
		\[
		\Phi_p(u, x_0, \rho) = \fint_{B_\rho(x_0)} |V(Du(x)) - V((Du)_\rho)|^2 \, dx \le c \left[ \fint_{B_\rho(x_0)} \left| V(Du(x)) - V(z_0) \right|^2 \, dx + \fint_{B_\rho(x_0)} \left| V(Du_\rho) - V(z_0) \right|^2 \, dx \right].
		\]
		The first summand tends to 0 by hypothesis. The second summand is equivalent to:
		\[
		\varphi_{|z_0|} \left( \left| \fint_{B_\rho(x_0)} Du(x) - z_0 \, dx \right| \right).
		\]
		By Jensen's inequality, we obtain:
		\[
		\varphi_{|z_0|} \left( \left| \fint_{B_\rho(x_0)} Du(x) - z_0 \, dx \right| \right) \le \fint_{B_\rho(x_0)} \varphi_{|z_0|} \left( |Du(x) - z_0| \right) \, dx.
		\]
		The second member of this inequality is equivalent to:
		\[
		\fint_{B_\rho(x_0)} \left| V(Du(x)) - V(z_0) \right|^2,
		\]
		which vanishes by hypothesis.
	\end{proof}
\end{prop}

Finally, we can prove the excess decay:

\begin{prop}
	Let us assume \( f \), \( \varphi \), and \( \psi \) satisfy hypotheses \((H.0)\) through \((H.4)\) for given \( p_1, q_1 \), and \( M \). Choose any \( L > M + 1 > 0 \), \( \alpha \in (0, 1) \), and \( z_0 \in \mathbb{R}^{nN} \) such that \( L > |z_0| > M + 1 \). Then, there exist constants \( \varepsilon_0 > 0 \), \( \theta \in (0, 1) \), and a radius \( \rho^* > 0 \) depending on \( n, N, L, p_1, q_1, \Gamma, \alpha, \gamma, x_0, z_0 \), and \( \Lambda_L := \max_{B_{L+2}} |D^2f| \), and with \( \varepsilon_0 \) also depending on \( \omega_L \), such that the following result holds. 
	
	Let us consider \( u \) a \( W^{1,\varphi} \)-minimizer of \( \mathcal{F} \) on \( B_\rho(x_0) \), with \( \rho < \rho^* \) and \( x_0 \in \mathbb{R}^n \) satisfying:
	\[
	\lim_{\rho \to 0} \fint_{B_\rho(x_0)} \left| V(Du(x)) - V(z_0) \right|^2 \, dx = 0.
	\]
	If:
	\begin{equation} \label{unouno}
	\Phi_\varphi(u, x_0, \rho) \le \varepsilon_0,
	\end{equation}
	then:
	\[
	\Phi_\varphi(u, x_0, \theta \rho) \le \theta^{2\alpha} \Phi_\varphi(u, x_0, \rho).
	\]
\end{prop}

\begin{proof}
Let \( z_0 \) be such that \( |z_0| > M + 1 \), and let \( x_0 \) be any point such that  
\[
\lim_{\rho \to 0} \fint_{B_\rho(x_0)} |V(Du(x)) - V(z_0)|^2 \, dx = 0.
\]
For simplicity of notation, we assume \( x_0 = 0 \) and define  
\[
z = (Du)_\rho := \fint_{B_\rho} Du \, dx
\]
and  
\[
\Phi_\varphi(\cdot) := \Phi_\varphi(u, 0, \cdot).
\]
Here, \( \rho > 0 \) is any positive value small enough (smaller than a \( \rho^* \) that will be determined throughout the proof).  
Since the claim is trivial if there exists a \( \rho \) such that \( \Phi_\varphi(\rho) = 0 \), we assume that \( \Phi_\varphi(\rho) \neq 0 \).

We now define  
\[
w(x) := u(x) - zx,
\]  
and using our equivalent definition of \( \Phi_\varphi(\rho) \), we have  
\[
\fint_{B_\rho} \varphi_{|z|}(|Dw|) \, dx = \Phi_\varphi(\rho).
\]  
Next, we approximate by \( \mathcal{A} \)-harmonic functions, where \( \mathcal{A} := D^2f(z) \).

If \( \rho \) is sufficiently small, we have \( L > |z| > M + 1 \). Hence, from \( |\mathcal{A}| \le \max_{B_{L+2}} |D^2f| =: \Lambda_L \), and by Lemma \ref{duepunticinque}, we deduce that \( \mathcal{A} \) satisfies \eqref{Legendre} with ellipticity constant \( 2\gamma \).  
Lemma \ref{diecibis} yields the estimate

\[
\left| \fint_{B_\rho} \mathcal{A}(Dw, D\phi) \, dx \right| \le C_2 \sqrt{\Phi_\varphi(\rho)} \omega_L\left( \Phi_\varphi(\rho) \right) \sup_{B_\rho} |D\phi|
\]
for all \( \rho < \rho^* \) and for all smooth functions \( \phi : B_{\rho} \to \mathbb{R}^N \) with compact support in \( B_\rho \), where \( C_2 \) is a constant depending on \( n, N, p_1, q_1, \Gamma, L, \Lambda_L \).

For \( \varepsilon > 0 \) to be specified later, we fix the corresponding constant \( \delta(n, N, \varphi, \Lambda_L, \gamma, \varepsilon) > 0 \) from Lemma \ref{Armonicbis}.  
Let \( \varepsilon_0 = \varepsilon_0(n, N, \varphi, \Lambda_L, \gamma, \varepsilon) \) be small enough so that \eqref{unouno} implies  

\begin{equation} \label{settepuntoquindici}
	C_2\omega_{L}(\Phi_\varphi(\rho))\le \delta
\end{equation}
\begin{equation}\label{settepuntosedici}
	\kappa=\sqrt{\Phi_\varphi(\rho)}\le 1.
\end{equation}

We apply Lemma \ref{Armonicbis}, obtaining an \( \mathcal{A} \)-harmonic function \( h \in C^\infty_{\text{loc}}(B_\rho; \mathbb{R}^N) \) such that  
\[
\sup_{B_{\rho/2}} |Dh| + \rho \sup_{B_{\rho/2}} |D^2h| \le c^* \varphi_{|z|}^{-1} \left( \Phi_\varphi(\rho) \right),
\]  
where \( c^* = c^*(n, N, \varphi, \Lambda_L, \gamma) \), and  
\[
\fint_{B_{\rho/2}} \varphi_{|z|} \left( \frac{|w - h|}{\rho} \right) \, dx \le \varepsilon \left[ \Phi_\varphi(\rho) + \varphi_{|z|}(\kappa) \right] \le c \varepsilon \Phi_\varphi(\rho),
\]  
where this step follows from the fact that \( \varphi_{|z|}(t) \simeq t^2 \) when \( t < 1 \).

Now fix \( \theta \in (0, 1/4] \). A Taylor expansion gives the estimate  
\[
\sup_{x \in B_{2\theta \rho}} |h(x) - h(0) - Dh(0) x| \le \frac{1}{2} (2\theta \rho)^2 \sup_{x \in B_{\rho/2}} |D^2h| \le 2 c^* \theta^2 \rho \varphi_{|z|}^{-1} \left( \Phi_\varphi(\rho) \right).
\]
Thus, we obtain  
\[
\fint_{B_{2\theta \rho}} \varphi_{|z|} \left( \frac{|w(x) - h(0) - Dh(0) x|}{2\theta \rho} \right) \, dx \le  
c \left[ \theta^{-q_1 - 1} \fint_{B_{\rho/2}} \varphi_{|z|} \left( \frac{|w - h|}{\rho} \right) \, dx + \fint_{B_{2\theta \rho}} \varphi_{|z|} \left( \frac{|h(x) - h(0) - Dh(0) x|}{2\theta \rho} \right) \, dx \right] \le  
c \left[ \theta^{-q_1 - 1} \varepsilon \Phi_\varphi(\rho) + \varphi_{|z|}(\theta \kappa) \right].
\]
This simplifies to  
\[
c \left[ \theta^{-q_1 - 1} \varepsilon \Phi_\varphi(\rho) + \theta^2 \Phi_\varphi(\rho) \right] \le c \theta^2 \Phi_\varphi(\rho).
\]
Thus,  
\begin{equation}\label{ciccio}
\fint_{B_{2\theta \rho}} \varphi_{|z|} \left( \frac{|u(x) - zx - (h(0) + Dh(0) x)|}{2\theta \rho} \right) \, dx \le c \theta^2 \Phi_\varphi(\rho).
\end{equation}
On the other hand, using the definition of \( s \) and the properties of \( h \), we have  
	\begin{equation} \label{settepuntodiciannove}
	|Dh(0)|^2\le (c^*)^2 \left[\varphi^{-1}_{|z|}\left(\Phi_{\varphi}(\rho)\right)\right]^2.
\end{equation}
Choosing \( \varepsilon_0 \) small enough so that \eqref{unouno} implies  
	\begin{equation} \label{hpiccolo}|Dh(0)|^2\le 1,
\end{equation}
 using also \eqref{settepuntodiciannove}, we get 
\begin{multline} \label{ciccio2}
	\Phi_{\varphi}(2\theta\rho,z+Dh(0))\le\\ \le c\left[(2\theta)^{-n}\left(\fint_{B_\rho}|V(Du(x))-V(z)|^2\,dx+\varphi_{|z|}(|Dh(0)|)\right)\right]\le\\ \le c\left[\theta^{-n}\left(\Phi_\varphi(\rho)+\Phi_\varphi(\rho)\right)\right] \le c\theta^{-n}\Phi_\varphi(\rho).
\end{multline}
Now, combining the Caccioppoli inequality \eqref{caccioppolibis}, \eqref{ciccio} and \eqref{ciccio2} choosing \( \zeta = h(0) \) and \( z + Dh(0) \) in place of \( z \), we obtain  
\[
\Phi_\varphi(\theta \rho, z + Dh(0)) \le c \left[ \theta^2 \Phi_\varphi(\rho) + \theta^{2\beta} \Phi_\varphi(\rho)^\beta + \theta^{-n\beta} \Phi_\varphi(\rho)^\beta \right].
\]
	Now, we can combine \eqref{ciccio} and \eqref{ciccio2} and Caccioppoli inequality \eqref{caccioppolibis} with $\zeta=h(0)$ and $z+Dh(0)$ instead of $z$, and we get:
\begin{equation} \label{settepuntoventuno}
	\Phi_\varphi(\theta\rho,z+Dh(0))\le c\left[\theta^2\Phi_\varphi(\rho)+\theta^{2\beta}\Phi_\varphi(\rho)^{\beta}+\theta^{-n\beta}\Phi_\varphi(\rho)^{\beta}\right].	\end{equation}
Thereby the condition $|z+Dh(0)|\le L+1$ of Lemma \ref{caccioppolilemmabis} can be deduced from \eqref{hpiccolo}.\\
From the condition \( |z + Dh(0)| \le L + 1 \) (which follows from \eqref{hpiccolo}), we can deduce this inequality.  
If \( \varepsilon_0 \) is chosen small enough, depending on \( \theta \), \eqref{unouno} implies  
\begin{equation}\label{settepuntoventitre}
\theta^{-n\beta}\Phi_\varphi(\rho)^{\beta-1}\le \theta^2,
\end{equation}
and from the fact that $\theta\le 1$ we have
$$\Phi_\varphi(\theta\rho,z+Dh(0))\le c\theta^2\Phi_\varphi(\rho).$$
Adapting Lemma 6.2 in \cite{Schmidt} (it just uses simples ideas like the ones from Proposition \ref{rpfopsemplice}) we deduce, from \eqref{settepuntoventitre}:
\begin{equation}\label{settepuntoventiquattro}
\Phi_\varphi(\theta\rho)\le C_3\theta^2\Phi_\varphi(\rho),
\end{equation}
where $C_3>0$ depends on $n,N,\varphi,\Gamma,\gamma,\Lambda_L,L$.\\ Finally, choosing \( \theta \in (0, 1/4] \)  (depending on $\alpha$ and whatever $C_3$ depends on) small enough so that  
 \begin{equation}\label{settepuntoventicinque}
C_3\theta^2\le \theta^{2\alpha}
\end{equation}
holds, and $\varepsilon_0$ small enough such that \eqref{settepuntoquindici}, \eqref{settepuntosedici}, \eqref{settepuntoventitre} follow from \eqref{unouno}. Taking into account \eqref{settepuntoventiquattro} and \eqref{settepuntoventicinque} the proof of the proposition is complete.
\end{proof}

The following adaptation of (\cite{Schmidt}, Lemma 7.10) is a straightforward consequence of repeatedly applying the previous proposition.

\begin{thm} \label{iterato}
	Let us assume that \(f\), \(\varphi\), and \(\psi\) satisfy hypotheses \((H.0)\) through \((H.4)\) for the given values of \(p_1\), \(q_1\), and \(M\).\\
	Choose any \(L > 2M + 2 > 0\), \(\alpha \in (0,1)\), and \(z_0 \in \mathbb{R}^{nN}\) such that \(\frac{L}{2} > |z_0| > M + 1\).\\
	Then, there exist a constant \(\tilde{\varepsilon}_0 > 0\) and a radius \(\rho^* > 0\), depending on \(n\), \(N\), \(L\), \(p_1\), \(q_1\), \(\Gamma\), \(\alpha\), \(\gamma\), \(x_0\), \(z_0\), and \(\Lambda_L := \max\limits_{B_{L+2}} |D^2f|\), where \(\tilde{\varepsilon}_0\) also depends on \(\omega_L\), such that the following holds.\\
	Let \(u\) be a \(W^{1,\varphi}\)-minimizer of \(\mathcal{F}\) on \(B_\rho(x_0)\), with \(\rho < \rho^*\) and \(x_0 \in \mathbb{R}^n\) satisfying
	\[
	\lim\limits_{\rho \to 0} \fint_{B_\rho(x_0)} |V(Du(x)) - V(z_0)|^2 \, dx = 0.
	\]
	If
	\[
	\Phi_\varphi(u,x_0,\rho) \le \tilde{\varepsilon}_0,
	\]
	then there exists a constant \(c\) depending on \(n\), \(N\), \(L\), \(p_1\), \(q_1\), \(\Gamma\), \(\alpha\), \(\gamma\), \(x_0\), and \(z_0\) such that
	\[
	\Phi_\varphi(u,x_0,r) \le c \left( \frac{r}{\rho} \right)^{2\alpha} \Phi_\varphi(u,x_0,\rho)
	\]
	for any \(r < \rho\).
\end{thm}

\begin{proof}
	The theorem stated in the introduction follows from Campanato's integral characterization of H\"older continuity.
\end{proof}

\section{Acknowledgements}
The author is supported by MathInParis project by Fondation Sciences mathématiques de Paris (FSMP), funding from the European Union’s Horizon 2020 research and innovation programme, under the Marie Skłodowska-Curie grant agreement No 101034255. \\
She is also supported by Sorbonne Universite, being affiliated at Laboratoire Jacques Louis Lions (LJLL). \\
During the writing of this paper, she was also a $GNAMPA$ member. \\ $GNAMPA$ supported her with "INdAM -GNAMPA Project 2024", codice CUP E53C23001670001, "Non-smooth optimal control problems" (coordinator: Francesca Angrisani). \\

\end{document}